\documentclass[11pt,reqno]{amsart}
\usepackage[usenames]{color}
\usepackage{amsmath,amssymb,pdfsync,verbatim,graphicx,epstopdf,enumerate,fancybox}
\usepackage[notref,notcite]{}
\newtheorem{theorem}{Theorem}[section]
\newtheorem{lemma}[theorem]{Lemma}

\newtheorem{corollary}[theorem]{Corollary}

\newtheorem{remark}[theorem]{Remark}

\numberwithin{equation}{section}
\newtheorem{example}{Example}[section]

\newtheorem{question}{Question}

\author{Ramya Dutta and Suman Kumar Sahoo}

\newcommand{\PD}{\partial}

\newcommand{\Beq}{\begin{equation}}
\newcommand{\Eeq}{\end{equation}}
\newcommand{\beq}{\begin{equation*}}
\newcommand{\eeq}{\end{equation*}}
\newcommand{\bal}{\begin{align}}
\newcommand{\eal}{\end{align}}

\newcommand{\bp}{\begin{prob}}
	\newcommand{\ep}{\end{prob}}
\newcommand{\bpr}{\begin{proof}}
	\newcommand{\epr}{\end{proof}}

\usepackage{amsmath,tikz}
\usetikzlibrary{matrix}

\newcommand*\Rn{\mathbb{R}^n}

\newcommand{\tn}{T\mathbb{S}^{n-1}}

\usepackage[centerlast,small,sc]{caption}
\title{Symmetry from sectional integrals for convex domains}

\begin{document}
	% Please provide minimum  5 keywords.
	\email{ramya@tifrbng.res.in,suman@tifrbng.res.in}
	\address{$^\ast$Corresponding author
		\newline\indent$^\dagger$TIFR Centre for Applicable Mathematics, Sharada Nagar, Chikkabommasandra,\newline\indent\hspace{0mm} Yelahanka New Town, Bangalore, India}
	
	\begin{abstract}
		Let $\Omega$ be a bounded convex domain in $\mathbb{R}^n$ ($n \ge 2$). In this work, we prove that if there exists an integrable function $f$ such that it's Radon transform over $(n-1)$-dimensional hyperplanes intersecting the domain $\Omega$ is a strictly positive function of distance to the nearest parallel supporting hyperplane to $\Omega$, then $\Omega$ is a ball and the function $f$ is a unique radial function about the center of $\Omega$.
	\end{abstract}

	\maketitle
	
	\section{Introduction}
	
	Let $\Omega$ be a domain in $\mathbb{R}^n$ and $ f $ be an integrable function in $ \Omega$. We extend $ f $ to be identically $0$ outside of $ \Omega $. Let $ \langle ,\rangle $ denotes the usual inner product in $\Rn$. Throughout the article $\mathcal{H}^k \llcorner G$ denotes the $k$-dimensional Hausdorff measure restricted to $G$, a Borel measurable subset of $\mathbb{R}^n$ for $1 \le k \le n$. 
	
	The ray transform integrates scalar functions over straight lines. The family of oriented lines can be parametrized by the points on the manifold
	\begin{align*}
	\tn=\{(x,\xi)\in \Rn \times \Rn: \langle x,\xi \rangle=0, |\xi|=1 \} \subset  \Rn \times \Rn
	\end{align*}
	which is the tangent bundle of the unit sphere. The ray transform $I$ of an integrable function $f$ is a function defined on $\tn$ as
	\begin{align}\label{defiray}
	If(x,\xi)=\int_{-\infty}^{+\infty}f(x+t\xi) \,dt.
	\end{align}
	
	The Radon transform integrates the functions over the hyperplanes. The Radon transform $R$ of a function $f \in L^1(\Omega)$ is the function defined on $\mathbb{S}^{n-1} \times \mathbb{R}$ by 
	\begin{align}\label{defiradon}
	Rf(w,p)=\int_{\Sigma_{\omega,p}} f(x)\,d \mathcal{H}^{n-1}\llcorner \Sigma_{\omega,p}
	\end{align}
	where $\displaystyle \Sigma_{w,p}=\{ x\in \Rn | \langle x,\omega\rangle=p \} $ denotes the hyperplane with $p$ as the perpendicular distance from the origin and $\omega$ is normal to the plane. $\displaystyle If|_{\Omega}$ denotes the ray transform of $f$ along all the lines intersecting the domain $\Omega$. Similar definition stands for the notation $\displaystyle Rf|_{\Omega}$. The operators \eqref{defiray} and \eqref{defiradon} have  been well studied and has many applications in computer tomography. For more detailed study of the operators $I$ and $R$ we refer \cite{Helgason,Sh}.
	
	In dimension $n=2$ both operators $I$ and $R$ coincide. The characterization of range of these operators have been well studied in  case of Schwartz class functions $ \mathcal{S}(\Rn) $. It is known that $I$ and $R$ are linear isomorphisms between the Schwartz spaces $\mathcal{S}(\mathbb{R}^n)$ to $\mathcal{S}(\tn)$ and $\mathcal{S}(\mathbb{S}^{n-1} \times \mathbb{R})$ respectively.  Both the operators $I$ and $R$ are injective, i.e., if $If=0$ imply $f=0$ and $Rf=0$ imply $f=0$.
	
	The question one is interested in is,
	
	 \begin{question} Are non-zero constant functions in the range of ray transforms, i.e., does $\exists \, f \in L^1(\Omega)$ such that $ If|_{\Omega} = c \, (\neq 0)$? \end{question}
	
	We came to know about this problem when second named author in this paper attended a conference talk by Joonas Ilmavirta titled ``Functions of constant $X$-ray transform" based on his joint work with Gabriel Paternain in University of Jyv\"{a}skyl\"{a}, Finland on ``Inverse problems: PDE and Geometry", where they gave a positive answer to the above question (see \cite{IP}). They showed that in Euclidean spaces of dimension $n\geq 2$, strictly convex domains $\Omega$ admitting functions of constant $X$-ray transform are balls and the function is unique and radial. They also considered the analogue of the problem on Riemannian manifolds. The function $\displaystyle f(x) = \frac{\chi_{\{|x| < R\}}}{\pi\sqrt{R^2-|x|^2}}$ in the ball of radius $R$ centered at origin of $\mathbb{R}^n$ has constant X-ray transform in its support in $\mathbb{R}^n$.
	
	Similarly, one can ask the same question for the Radon transform in dimensions $n \ge 2$.
	
	\begin{question} Are non-zero constant functions in the range of Radon transforms in dimension $n \ge 2$, i.e., does $\exists \, f \in L^1(\Omega)$ such that $ Rf|_{\Omega} = c \, (\neq 0)$? \end{question}
	
	We address this question in section $2$ and show that one cannot expect constants to be in the range of Radon transform in dimensions $ n\ge 3 $. Instead we prove the following analogue. If $\Omega$ is a bounded convex domain in $\mathbb{R}^n$ ($n \ge 2$) and if $\exists \, f \in L^1(\Omega)$ such that $Rf|_{\Omega}$ is a strictly positive function of distance to nearest parallel supporting hyperplane to $\Omega$ alone [see Theorem \ref{mainthm} for precise definition], then $\Omega$ is a ball and $f$ is a unique radial function about the center of $\Omega$.
	
	The proofs utilize the zeroth and the first order moments of the function which was originally used in context of ray transform in the plane by Joonas Ilmavirta \& Gabriel Paternain.
	
	Subsequently, as a corollary to the main result in our paper, we obtain the result for the ray transform $If|_{\Omega} = c \, (\neq 0)$ in Euclidean space under milder assumptions requiring $\Omega$ to be a bounded domain in $\mathbb{R}^2$ or a bounded convex domain in $\mathbb{R}^n$ for $n \ge 3$. To our knowledge, the main result in Section 2 is completely new.
	
	\begin{remark}
		If $ \Omega $ is bounded convex domain in $\mathbb{R}^n$, any line $\ell$ intersecting $\Omega$ intersects $\PD\Omega$ at exactly two points. Also if $p \in \partial \Omega$ then there exists at least one supporting $(n-1)$-dimensional hyperplane of $\Omega$ passing through $p$.
	\end{remark}
	
	\section{radon transform as a distance function}
	
	We begin by stating the main result.
	
	\begin{theorem}\label{mainthm} Let, $\Omega \subset \mathbb{R}^n$ be a bounded convex domain ($n \ge 2$) and $G:[0,\infty) \to \mathbb{R}$ be a strictly positive locally integrable function. If, $\exists \, f : \Omega \to \mathbb{R}$ an integrable function such that, \begin{align}\label{inteplane} 
		\int_{\Omega \, \cap \, \Sigma} f(x)\,d\mathcal{H}^{n-1} \llcorner \Sigma = G\left(\min_{j=1,2}\operatorname{dist}\left(\Sigma,\Pi_j\right)\right) \end{align} for all $(n-1)$-dimensional hyperplanes $\Sigma$ satisfying $\mathcal{H}^{n-1}(\Omega \cap \Sigma) > 0$ and $\Pi_{j}$ (for, $j=1,2$) are the pair of supporting $(n-1)$-dim hyperplanes to $\Omega$ that are parallel to $\Sigma$, then $\Omega$ is a ball and $f$ is a unique radial function. \end{theorem}
	
	\begin{remark} Before going into the proof of Theorem \ref{mainthm} let us address Question $2$.  Suppose, $\exists \, f \in L^1(\Omega)$ such that $Rf|_{\Omega} = 1$. In view of Theorem \ref{mainthm}, letting $G \equiv 1$ we see that $\Omega$ has to be a ball, without loss of generality assume $\Omega = B(0,R) \subset \mathbb{R}^n$ for some $R > 0$. Let us consider the implication of this result on the Fourier slices of the function $f$. Given $\xi \in \mathbb{S}^{n-1}$, let us choose an orthogonal coordinate frame for the space $\mathbb{R}^n$ s.t., $x = (t,x') \in \mathbb{R}_{\xi} \times \mathbb{R}^{n-1}_{\xi^\perp}$ and use the notation $\displaystyle \Sigma^t_\xi$ to denote the $(n-1)$-dimensional hyperplane having $\xi \in \mathbb{S}^{n-1}$ as the normal and passing through the point $(t,0') \in \mathbb{R}_\xi \times\mathbb{R}^{n-1}_{\xi^\perp}$. Now, by Fubini's theorem we have, \begin{align*} \mathcal{F}(f)(\xi) = \int_{\mathbb{R}^n} e^{-i\left<x,\xi\right>}f(x)\,dx = \int_{\mathbb{R}_\xi} e^{-i|\xi| t}\left( \int_{\mathbb{R}_{\xi^\perp}^{n-1}} f(t,x')\,d \mathcal{H}^{n-1} \llcorner \Sigma_{\xi}^t(x')\right)\,dt = \int_{-R}^{R} e^{-i|\xi| t}\,dt = \frac{2\sin(R|\xi|)}{|\xi|}. \end{align*} where, in the third equality we used hypothesis (\ref{inteplane}) with $G \equiv 1$.
	
	It is known that for $n \ge 2$, \begin{align*} \mathcal{F}^{-1}\left(\frac{e^{it|\xi|}}{|\xi|}\right)(x) = \begin{cases} \dfrac{1}{2\pi}\operatorname{sgn}(t) \chi_{\{|x| < |t|\}} (t^2 - |x|^2)^{-\frac{1}{2}}, & \text{ when, } n = 2 \\ & \\ C_n\lim\limits_{\epsilon \downarrow 0} \left(|x|^2 - (t+i\epsilon)^2\right)^{-\frac{n-1}{2}}, & \text{ when, } n \ge 3 \end{cases} \end{align*} where, $\displaystyle C_n=\frac{\Gamma\left(\frac{n+1}{2}\right)}{(n-1)\pi^{\frac{n+1}{2}}}$ and $t \in \mathbb{R}$, which is integrable when $ n=2 $ but certainly not in $L^1(\mathbb{R}^n)$ for $n \ge 3$. Therefore, there does not exist integrable functions such that $Rf|_\Omega=c(\ne 0) $ in dimensions $n \ge 3$.
	\end{remark}
	    
	    \begin{proof}[Proof of Theorem \ref{inteplane}]
	    
	    Let $\Omega$ be a bounded convex domain in $\mathbb{R}^n$ ($n \ge 2$).	
		Let us fix an orthogonal coordinate system such that the origin is equidistant from each pair of parallel supporting $(n-1)$-dimensional hyperplanes to $\Omega$ sharing $x_j$-axis as their common normal, for $j = 1(1)n$.  Let us extend the function $f$ in $\Omega^c$ by $0$ and consider the restriction of first moment of $f$ on the unit sphere in $\mathbb{R}^n$, \begin{align} g(\xi) &:= \int_{\mathbb{R}^n} \left<x,\xi\right>f(x)\,dx, \, \text{ for } \, \xi \in \mathbb{S}^{n-1} \\&= \int_{\mathbb{R}^{n-1}_{\xi^\perp}}\int_{\mathbb{R}_{\xi}} t f(t, x')\,dt\,d\mathcal{H}^{n-1} \llcorner \Sigma^t_\xi(x') \end{align} where, we decomposed $x = (t,x') \in \mathbb{R}_{\xi} \times \mathbb{R}^{n-1}_{\xi^\perp}$ and in our notation $\displaystyle \Sigma^t_\xi$ denotes the $(n-1)$-dimensional hyperplane having $\xi \in \mathbb{S}^{n-1}$ as the normal and passing through the point $(t,0') \in \mathbb{R}_\xi \times R^{n-1}_{\xi^\perp}$. Let  $\Pi_j^\xi$ for $j=1,2$ be the pair of supporting $(n-1)$-dim hyperplanes of $\Omega$ orthogonal to $\xi$, so that if $\displaystyle \, r_1(\xi) := \inf\left\{ t\in \mathbb{R}_{\xi} : \mathcal{H}^{n-1}(\Sigma^{t}_\xi \cap \Omega) > 0\right\}$ and $\displaystyle \, r_2(\xi) := \sup\left\{ t\in \mathbb{R}_{\xi} : \mathcal{H}^{n-1}(\Sigma^{t}_\xi \cap \Omega) > 0\right\}$ then $\displaystyle \Pi_j^\xi = \Sigma_{\xi}^{r_j(\xi)}$ for $j=1,2$.  Now, by given condition \eqref{inteplane} we have, \begin{align} \int_{\mathbb{R}^{n-1}_{\xi^\perp}} f(t, x')\,d\mathcal{H}^{n-1} \llcorner \Sigma^t_\xi(x') = G\left(\min_{j=1,2}\operatorname{dist}\left(\Sigma^t_\xi,\Pi^\xi_j\right)\right)\end{align} for  $t \in \mathbb{R}_{\xi}$ such that the plane $\Sigma^t_\xi$ satisfies $\mathcal{H}^{n-1}(\Sigma^{t}_\xi \cap \Omega) > 0$.  Using equation $(2.4)$ in equation $(2.3)$ and writing $r_j = r_j(\xi)$ ($j = 1,2$) for brevity we have, \begin{align}\label{firstmoment}
		 \nonumber g(\xi) &= \int_{\mathbb{R}_{\xi}} t \left(\int_{\mathbb{R}^{n-1}_{\xi^\perp}} f(t, x') \,d\mathcal{H}^{n-1} \llcorner \Sigma^t_\xi(x')\right)\,dt \nonumber\\&= \int_{r_1}^{r_2} t \left(\int_{\mathbb{R}^{n-1}_{\xi^\perp}} f(t, x') \,d\mathcal{H}^{n-1} \llcorner \Sigma^t_\xi(x')\right) \,dt \nonumber\\&= \int_{\frac{r_2+r_1}{2}}^{r_2} t \, G(r_2-t) \,dt + \int_{r_1}^{\frac{r_2+r_1}{2}} t \, G(t-r_1) \,dt \nonumber\\&= \int_{0}^{\frac{r_2-r_1}{2}} (r_2-t) \, G(t) \,dt + \int_{0}^{\frac{r_2-r_1}{2}} (t+r_1) \, G(t) \,dt \nonumber\\&= (r_2(\xi) + r_1(\xi))\int_{0}^{\frac{r_2-r_1}{2}(\xi)} G(t) \,dt. \end{align} 
		
	     Similarly we have, \begin{align}\label{areacalcu} 
     	\nonumber	K = \int_{\Omega} f(x)\,dx &= \int_{\mathbb{R}^{n-1}_{\xi^\perp}}\int_{\mathbb{R}_{\xi}} f(t, x')\,dt\,d\mathcal{H}^{n-1} \llcorner \Sigma^t_\xi(x') \nonumber\\&= \int_{r_1(\xi)}^{r_2(\xi)} \left(\int_{\mathbb{R}^{n-1}_{\xi^\perp}} f(t, x') \,d\mathcal{H}^{n-1} \llcorner \Sigma^t_\xi(x')\right)\,dt \nonumber\\&= 2\int_{0}^{\frac{r_2-r_1}{2}(\xi)} G(t) \,dt, \, \forall \, \xi \in \mathbb{S}^{n-1}.\end{align} 
	
		Therefore, $\displaystyle g(\xi) = \frac{K}{2}(r_2(\xi) + r_1(\xi))$. We note that by our choice of origin we have $r_2(\pm e_j) + r_1(\pm e_j) = 0$ and hence $g(\pm e_j) = 0$ for each $j = 1(1)n$ where, $e_j$ is the unit vector along $x_j$-axis. By definition $g$ is restriction of a linear function on $\mathbb{S}^{n-1}$ and $g(\pm e_j) = 0$ for $j=1(1)n$. Therefore, $g \equiv 0$ in $\mathbb{S}^{n-1}$ and hence, $r_2(\xi) = -r_1(\xi), \, \forall \, \xi \in \mathbb{S}^{n-1}$.
		
		Now, the function $\displaystyle u(r) = \int_0^r G(t)\,dt$ is injective (since, $G > 0$ by hypothesis). From equation $(2.6)$ we have, $\displaystyle u\left(\frac{r_2(\xi)-r_1(\xi)}{2}\right) = \frac{K}{2}$, for all $\xi \in \mathbb{S}^{n-1}$, therefore we conclude $r_2(\xi) - r_1(\xi) = 2R$ (for some positive constant $R$). That is the domain $\Omega$ has constant width in the sense that the distance between any pair of parallel $(n-1)$-dim supporting hyperplanes of $\Omega$, with common normal $\xi \in \mathbb{S}^{n-1}$, is a constant and equals $2R$.
		
		Combining these two facts we see that $|r_2(\xi)| = |r_1(\xi)| = R$ for all $\xi \in \mathbb{S}^{n-1}$, i.e., the distance of any $(n-1)$-dim supporting hyperplane of $\Omega$ from origin is constant and equals $R$. 
		
		Now, by convexity of $\Omega$ it must lie in the intersection of all closed half-spaces $\overline{H}_\xi$ containing origin whose boundaries $ \partial H_\xi$  are  $(n-1)$-dimensional supporting hyperplanes of $\Omega$ (with $\xi \in \mathbb{S}^{n-1}$ as outward normal to $\partial H_\xi$). Since, $\operatorname{dist}(0, \partial H_\xi) = R$, we conclude $\displaystyle \Omega \subseteq \bigcap\limits_{\xi \in \mathbb{S}^{n-1}} \overline{H}_\xi = \overline{B(0,R)}$. On the other hand suppose $p \in \partial \Omega$ and $\Pi_p$ is a supporting $(n-1)$-dimensional hyperplane of $\Omega$ passing through $p$, then $\operatorname{dist}(0,p) \ge \operatorname{dist}(0,\Pi_p) = R$. Hence, $\operatorname{dist}(0,p) = R$ for all $p \in \partial \Omega$. Consequently, $\Omega$ is the ball $B(0,R)$ in $\mathbb{R}^n$.
		
		Now, $f$ is radial follows from uniqueness of Radon transform and the fact that $\Omega$ is a ball.  We observe that the function $f \circ A$, where $A$ is any orthogonal matrix, also satisfies hypothesis (\ref{inteplane}) and hence from uniqueness of Radon transform we conclude  $ f = f \circ A $, for all orthogonal matrices $A$. Therefore $f$ is a radial function. This completes the proof.
	\end{proof}
	
        Any integrable positive radial function in a ball in $\mathbb{R}^n$ satisfies the hypothesis (\ref{inteplane}). However, the following example is particularly interesting.
	
	\begin{example}	
		The functions $\displaystyle f(x) = \left(R^2 - |x|^2\right)^\gamma$ for $\gamma > -1$ satisfies the hypothesis in (\ref{inteplane}) with, \begin{align*} \int_{\Omega \, \cap \, \Sigma^d} f \,d\mathcal{H}^{n-1}\llcorner \Sigma^d = c_n(\gamma)(R^2 - d^2)^{\frac{n-1}{2} + \gamma}  \end{align*} where, $\Sigma^d$ is a $(n-1)$-dimensional hyperplane at a distance $d$ from the origin, therefore at a distance $(R-d)$ from the nearest parallel supporting hyperplane of $\Omega = B(0,R)$ and $G(R-d) = c_n(\gamma)(R^2 - d^2)^{\frac{n-1}{2} + \gamma}$ for, $0 \le d < R$. The constant $\displaystyle c_n(\gamma) = \frac{1}{2}(n-1)\alpha_{n-1} \int_0^1 r^{\frac{n-3}{2}}(1-r)^\gamma\,dr = \frac{1}{2}(n-1)\alpha_{n-1} \beta\left(\frac{n-1}{2},\gamma+1\right)$, (where, $\alpha_{n-1}$ is the volume of $(n-1)$-dim unit ball and $\beta(x,y)$ denotes the Beta function) depends only on the dimension $n$ and $\gamma$.
	\end{example}

    Now, we obtain the results for ray transform as corollaries to Theorem \ref{inteplane}.

    \begin{corollary} Let, $\Omega$ be a bounded domain in $\mathbb{R}^2$ or a bounded convex domain in $\mathbb{R}^n$ ($n \ge 3$). If, $\exists \, $ an integrable function $f: \Omega \to \mathbb{R}$ such    that, \begin{align}\label{interela}
	\int_{\Omega \, \cap \, \ell} f(x)\,d\,\mathcal{H}^1 \llcorner \ell  = 1 \end{align} for all ($1$-dim) lines $\ell$ with $\mathcal{H}^1(\Omega \cap \ell) > 0$, then $\Omega$ is a ball and $f$ is a unique radial function about the center of $\Omega$. \end{corollary}

    \begin{proof}
	\textbf{Case-I:} We start by analyzing the simplest case when $\Omega$ is a bounded convex domain in the $\mathbb{R}^2$. Then, as a consequence of Theorem \ref{inteplane} with $G \equiv 1$ we conclude that $\Omega$ is a disk in $\mathbb{R}^2$.
	
	The function $\displaystyle f(x) = \frac{\chi_{\{|x| < R\}}}{\pi \sqrt{R^2-|x|^2}}$ satisfies the hypothesis \eqref{interela} in the domain $\Omega = B(0,R) \subset \mathbb{R}^2$ and is the required radial function. Uniqueness of $f$ follows from injectivity of ray transform.
	
	Now, for a general bounded domain $\Omega \subset \mathbb{R}^2$, we note that requiring a line $\ell$ to intersect $\Omega$ i.e., $\mathcal{H}^1(\Omega \cap \ell) > 0$ is equivalent to $\mathcal{H}^1(\operatorname{Conv}(\Omega) \cap \ell) > 0$, where $\operatorname{Conv}(\Omega)$ denotes the convex hull of $\Omega$. Suppose, $\exists \, f \in L^1(\Omega)$ satisfying hypothesis \eqref{interela} and $\operatorname{Conv}(\Omega)\setminus \Omega \neq \phi$. Then, from the previous consideration, $\operatorname{Conv}(\Omega)$ is a ball and $f$ is as before, which would contradict that fact that $f \equiv 0$ on $\operatorname{Conv}(\Omega)\setminus \Omega$. Therefore, $\operatorname{Conv}(\Omega) = \Omega$.
	
	\textbf{Case-II:} Now, we address the case of bounded convex domains $\Omega$ in $\mathbb{R}^n$ when, $n \ge 3$.
	
	If, $f$ is a function satisfying \eqref{interela} in a convex set $\Omega \subset \mathbb{R}^n$, then every $2$-dimensional section of $\Omega$ i.e., $\Omega \cap \Pi$ is convex (where, $\Pi$ is a $2$-dimensional plane intersecting $\Omega$ s.t., $\mathcal{H}^2(\Omega \cap \Pi) > 0$) and admits the function $f$. From the previous case ($n=2$) we conclude that the $2$-dimensional section $\Omega \cap \Pi$ must be equivalent to a disk in $\mathbb{R}^2$. Therefore by Lemma \ref{lemma} the domain $\Omega$ admitting $f$ is a ball $B(p_c,R)$ (say) and $\displaystyle f(x) = \frac{\chi_{\{|x-p_c| < R\}}}{\pi \sqrt{R^2-|x-p_c|^2}}$ satisfying the hypothesis \eqref{interela} is the unique radial function. \end{proof}

    \begin{lemma}\label{lemma}
	Let, $n \ge 3$ and $\Omega$ be a bounded convex domain in $\mathbb{R}^n$ such that $\Omega \cap \Pi$ is equivalent to a disk in $\mathbb{R}^2$ for all $2$-dimensional planes $\Pi$ with $\mathcal{H}^2(\Omega \cap \Pi) > 0$, then $\Omega$ is a ball in $\mathbb{R}^n$. \end{lemma}

    \begin{proof} Let, $\ell_m$ be a maximal diametric line in $\Omega$ i.e., $\displaystyle \mathcal{H}^1(\Omega \cap \ell_m) = \sup_{\ell} \mathcal{H}^1(\Omega \cap \ell) = 2R$ (say) and let us denote the midpoint of $\ell_m \cap \partial \Omega$ with $p_c$. Let $p$ be any point on $\partial \Omega \setminus \{\ell_m \cap \partial \Omega\}$ and $\Pi_p$ be a $2$-dimensional plane passing through $p$ and containing the line $\ell_m$. Then, $\Pi_p$ must intersect the domain $\Omega$ in a set equivalent to a $2$-dimensional disk of radius $R$ by the maximality of $\mathcal{H}^1(\Omega \cap \ell_m)$. Therefore, $\operatorname{dist}(p_c,p) = R$ for all $p \in \partial \Omega$, i.e., $\Omega$ is the ball $B(p_c,R)$ in $\mathbb{R}^n$. This completes the proof of the lemma. \end{proof}

    In case $n \ge 3$ and $\Omega$ is a strictly convex domain, we note that information on lines close enough to the boundary of $\Omega$ is sufficient to conclude that $\Omega$ is a ball.

    \begin{corollary} Let, $n \ge 3$ and $\Omega$ be a strictly convex domain. Suppose, for every $p \in \partial \Omega$ there is a neighbourhood $\mathcal{U}_p \, (\subset \mathbb{R}^n)$ of $p$ s.t., hypothesis \eqref{interela} holds for all lines $\ell$ with $\Omega \cap \ell \subset \mathcal{U}_p$ and $\mathcal{H}^1(\Omega \cap \ell) > 0$. Then, $\Omega$ is a ball. \end{corollary}

    \begin{proof} Arguing as before ($n=2$ case) we have $\Omega \cap \Sigma$ is equivalent to an open disk in $\mathbb{R}^2$ for every $2$-dimensional plane $\Sigma$ s.t., $\Omega \cap \Sigma \subset \mathcal{U}_p$ and $\mathcal{H}^2(\Omega \cap \Sigma) > 0$. Therefore, the Dupin's Indicatrix (see \cite[pp. 363-365]{Coxeter}) of $p$ along any $2$-dimensional plane spanned by a pair of principal directions is a circle, meaning $p$ must be umbilical. Therefore, boundary of $\Omega$ being all umbilical must be a sphere. 
    \end{proof}
	
	\noindent
	{\bf Acknowledgements:} 
		The authors wish to thank Joonas Ilmavirta for his conference  presentation at University of Jyv\"{a}skyl\"{a}, Finland which brought their attention to this problem. The authors would also like to thank Venkateswaran P. Krishnan, Sandeep Kunnath and Sivaguru Ravisankar for fruitful discussions and enlightening conversations which helped in improving the results.

\end{document}